\documentclass{amsart}

\newtheorem{theorem}{Theorem}[section]
\newtheorem{lemma}[theorem]{Lemma}
\newtheorem{proposition}[theorem]{Proposition}
\newtheorem{corollary}[theorem]{Corollary}
\newtheorem{assumption}[theorem]{Assumption}

\theoremstyle{definition}
\newtheorem{definition}[theorem]{Definition}

\newtheorem*{rep@theorem}{\rep@title}
\newcommand{\newreptheorem}[2]{%
\newenvironment{rep#1}[1]{%
 \def\rep@title{#2 \ref{##1}}%
 \begin{rep@theorem}}%
 {\end{rep@theorem}}}

\newreptheorem{theorem}{Theorem}

\theoremstyle{remark}
\newtheorem{remark}[theorem]{Remark}

\numberwithin{equation}{section}

\usepackage{cite}
\usepackage{graphicx}
\usepackage{pgfplots,tikz}
\usetikzlibrary{patterns}
\usepackage{appendix}
\usepackage[hidelinks]{hyperref}
\hypersetup{
    colorlinks=true,
    citecolor=[RGB]{183, 40, 46},
    linkcolor=[RGB]{183, 40, 46},
}
\usepackage{enumitem}

\newcommand{\setr}{\mathbb{R}}

\newcommand{\id}{\operatorname{Id}}

\newcommand{\bigo}{\mathcal{O}}

\newcommand{\pd}{\partial}

\newcommand{\supp}{\operatorname{supp}}

\newcommand{\cim}{\operatorname{Im}}
\newcommand{\cre}{\operatorname{Re}}

\newcommand{\abs}[1]{\lvert#1\rvert}

\begin{document}

\title[Transverse Geometric Control Gives Polynomial Decay]{Stabilisation of Waves on Product Manifolds by Boundary Strips}
\author{Ruoyu P. T. Wang}
\address{Department of Mathematics, Northwestern University, Evanston, Illinois 60208}
\email{rptwang@math.northwestern.edu}

\subjclass[2010]{35L05, 47B44}

\date{\today}

\keywords{damped wave, boundary stabilisation, polynomial decay, non-compact manifold, interior impedance problem, product manifold}

\begin{abstract}
We show that a transversely geometrically controlling boundary damping strip is sufficient but not necessary for $t^{-1/2}$-decay of waves on product manifolds. We give a general scheme to turn resolvent estimates for impedance problems on cross-sections to wave decay on product manifolds. 
\end{abstract}

\maketitle
\section{Introduction}
\subsection{Main Results}
Let $X$ be a compact smooth manifold with smooth boundary $\pd X$. Let $Y$ be a complete smooth manifold, not necessarily compact. Throughout the paper, compact manifolds means topologically compact ones with or without boundary. If $\pd Y\neq \emptyset$, we assume that the boundary condition on $\pd Y$ makes $\Delta_y\ge 0$ essentially self-adjoint on $L^2(Y)$. Let $\Gamma=\pd X\times Y$ and consider non-negative boundary damping functions $a,b\in L^{\infty}(\Gamma)$. We always impose an assumption:
\begin{assumption}\label{1t3}
We assume there is non-negative $a_0(x)\in C^{0}(\pd X)$ and constant $c_0>0$ that $c_0\le a(x,y)/a_0(x)\le 1$, $0\le b(x,y)/a_0(x)\le 1$ on $\Gamma$.
\end{assumption}
\begin{definition}[Geometric Control Condition]
We say $\Sigma\subset \pd X$ geometrically controls $X$, if there is $T>0$ that every unit-speed generalised geodesic hits $\Sigma$ non-diffractively in time $T$. This is the assumption of \cite[Theorem 5.5]{blr92}. 
\end{definition}

We consider the damped wave equation on the product manifold $\Omega=X\times Y$:
\begin{gather}\label{1l2}
(\pd_t^2+\Delta_x+\Delta_y)u(t,x,y)=0, \text{ on }\mathbb{R}_t\times \Omega,\\
(\pd_n+a(x,y)\pd_t+b(x,y))u(t,x,y)=0, \text{ on }\mathbb{R}_t\times\Gamma,\\
\label{1l4}
u(0,x,y)=u_0(x,y)\in H_\pd^2(\Omega),\ \pd_t u(0,x,y)=u_1(x,y)\in H_\pd^1(\Omega),
\end{gather}
where $H_\pd^s(\Omega)$ are the Sobolev spaces whose boundary condition on $X\times\pd Y$ is consistent with the chosen one on $\pd Y$. Define the energy of the solution by
\begin{equation}
E(u,t)=\frac12\int_\Omega \abs{\nabla u}^2+\abs{\pd_t u}^2. 
\end{equation}
\begin{theorem}[Transverse Geometric Control to Energy Decay]\label{0t1}
Assume $\{a_0(x)>0\}\subset \pd X$ geometrically controls $X$. Then there is $C>0$ such that
\begin{equation}\label{1l6}
E(u,t)^\frac12\le Ct^{-\frac1{2}}(\|u_0\|_{H^2(\Omega)}+\|u_1\|_{H^1(\Omega)}),
\end{equation}
either if $Y$ is compact, or if $0<c_0\le b/a_0$ on $\Gamma$. 
\end{theorem}

\begin{figure}
\begin{tikzpicture}[minimum size=0.01cm]
\draw [fill=gray!50] (-0.599,0.601) arc (45: 315: 0.85) -- (-0.670,-0.530) -- (-0.670, -0.530) arc (315: 45: 0.75) --cycle;
\draw (-1.2,0) circle (0.8);
\draw [dashed] (-1.2, 0.8) -- (0,0.8);
\draw [dashed] (-1.2, -0.8) -- (0,-0.8);
\node at (-1.2, 0) {$X$};


\fill[gray!25] (0, 0.8) arc (90:-90:-0.16 and 0.8) -- (3, -0.8) -- (3, -0.8) arc (-90:90:-0.16 and 0.8) -- (3,0.8) -- cycle;
\draw[fill=gray!50] (0, 0.8) arc (90:135:-0.16 and 0.8) -- (3.113, 0.566) -- (3.113, 0.566) arc (135:90:-0.16 and 0.8) -- (3,0.8) -- cycle;
\draw (0,0) ellipse (0.16 and 0.8);
\draw (0,0) ellipse (0.16 and 0.8);
\draw (0, -0.8) -- (3, -0.8);
\draw[fill=gray!50] (0, -0.8) arc (270:225:-0.16 and 0.8) -- (3.113, -0.566) -- (3.113, -0.566) arc (225:270:-0.16 and 0.8) -- (3,-0.8) -- cycle;

\draw (3, 0.8) arc (90:270:-0.16 and 0.8);
\draw [dashed] (3,0) ellipse (0.16 and 0.8);
\draw (0,0.8) -- (3,0.8); 


\node at (1.5,1) {$\Gamma$}; 

\draw [->] (0.5,-1) -- (2.5,-1);
\node at (1.5, -1.2) {$y$};
\end{tikzpicture}\hspace{3em}
\begin{tikzpicture}[minimum size=0.01cm]
\draw[fill=gray!50] (0, 0.75) -- (1.6, 0.75) -- (1.6, 0.85) -- (0, 0.85) -- cycle;
\draw (0, -0.8) -- (0, 0.8);
\draw (0, -0.8) -- (1.6, -0.8);
\draw (1.6, -0.8) -- (1.6, 0.8);
\draw (1.6, 0.8) -- (0, 0.8);
\node at (0.8, 1.05) {$\Gamma$};
\draw [->] (0.15,-1) -- (1.45,-1);
\node at (0.8, -1.2) {$y$};
\end{tikzpicture}\hspace{3em}
\begin{tikzpicture}[minimum size=0.01cm]
\draw [fill=gray!50](0, 0.6) -- (1.2, 0.6) -- (1.8,1.0) -- (0.6,1.0) -- cycle;
\draw (0, -0.6) -- (0, 0.6);
\draw (0, -0.6) -- (1.2, -0.6);
\draw (1.2, -0.6) -- (1.2, 0.6);
\draw (1.2, 0.6) -- (0, 0.6);
\draw (1.8,1.0) -- (1.8, -0.2) -- (1.2, -0.6);

\node at (0.8, 1.2) {$\Gamma$};
 \node at (0.8, -1.2) {};

\end{tikzpicture}
\caption{Three examples of $\Omega=X\times Y$ where Theorem \ref{0t1} applies, on $\mathbb{B}^2\times [0,1]$, $[0,1]\times \mathbb{S}^1$, $[0,1]\times\mathbb{T}^2$. The damping regions $\{a>0\}$ are shaded grey.} 
\label{f1}
\end{figure}
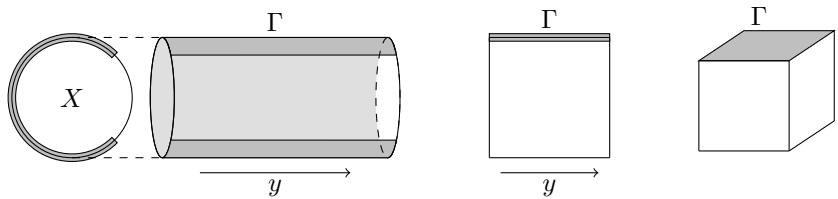

We give some examples in Figure \ref{f1}. One could also refer to Figure 4 to 6 in \cite{blr92} for more interesting choices of $X$ and $a_0(x)$, for example, Ikawa's bowling balls or dogbone regions. 

There have been some previous results on the decay of waves in cylindrical domains when the damping is uniformly positive on $\Gamma$: the $t^{-\delta}$-decay for some $\delta>0$ was shown in \cite{phu08}, and the $t^{-1/3}$-decay could be shown using the method in \cite{nis13}. In \cite{wan21} we showed the sharp $t^{-1/2}$-decay when $X$ is star-shaped. Theorem \ref{0t1} removes such star-shaped assumption. See Remark \ref{1t4}(3). This is also the only boundary damping result we know when $Y$ is non-compact, along the local energy decay obtained in \cite{roy18b}. 

When the damping could vanish on some part of $\Gamma$, the only result we know is an a priori $1/\log(2+t)$-decay, given by the Carleman estimates in \cite{lr97}. The $t^{-1/2}$-decay given in Theorem \ref{0t1} is more optimal on product manifolds. See also the references within \cite{wan21}. 

\begin{remark}\label{1t4}
\begin{enumerate}[wide]
\item The rate is sharp when $a_0>0$ on $\pd X$ and $a_0$ is sufficiently H\"{o}lder-regular. See Proposition \ref{4t2}.
\item The transverse geometric control is not necessary for the $t^{-1/2}$-rate. Let $X=[0,1]_x$, $Y=\mathbb{S}^1_y\times\mathbb{S}^1_z$ and $a(x)=1$. The damping region $\{a>0\}$ geometrically controls $[0,1]_x$ and hence one obtains $t^{-1/2}$-decay for waves on $X\times Y$: this is the third figure in Figure \ref{f1}. However, the problem is equivalent to that for $X'=[0,1]_x\times\mathbb{S}^1_z$, $Y'=\mathbb{S}^1_y$ and $a(x,z)=1$. In $X'$, the second figure of Figure \ref{f1}, the damping region $\{a>0\}=\{1\}_x\times \mathbb{S}^1_z$ does not geometrically control $X'\times Y'$, however the waves on $X'\times Y'=X\times Y$ enjoys the $t^{-1/2}$-decay. 
\item When $X$ is a bounded domain in $\mathbb{R}^d$ with smooth boundary, and $a_0(x)>0$ on $\pd X$, we have the geometric control condition, due to \cite[Lemma 5.3]{bsw16}.
\item The smoothness of $X$ could possibly be weakened to some manifolds with corners where the generalised bicharacteristic flows are well-defined, or $C^2$-metric with $C^3$-boundary: see \cite{bur97}. As long as we still have the spectral resolution of $\Delta_y$, $Y$ could be any manifold. See also Corollary \ref{1t2} for the case of $C^1$-boundary. 
\end{enumerate}
\end{remark}

We also obtain a general scheme to turn the resolvent estimates for interior impedance problems on the cross-section $X$ into energy decay rates of waves with boundary damping on $\Omega=X\times Y$:
\begin{theorem}[Transverse Resolvent Estimate to Energy Decay]\label{0t2}
Assume there exist $\delta\ge 0, \mu_0> 0$, $C>0$ such that uniformly for $\abs{\mu}\ge \mu_0$ and any $f\in L^2(X), g\in \sqrt{a_0}L^2(\pd X)$ and $u$ that
\begin{gather}\label{1l3}
(\Delta -\mu^2)u=f,\text{ on }X,\\
\label{1l5}
i\pd_n u+a_0\mu u=g,\text{ on }\pd X,
\end{gather}
we have
\begin{equation}\label{1l1}
\|u\|_{L^2(X)}^2+\langle \mu\rangle^{-2}\|\nabla u\|_{L^2(X)}^2\le C\langle \mu\rangle^{2\delta}\|f\|_{L^2(X)}^2+C\langle \mu\rangle^{-2+\delta}\|g/\sqrt{a_0}\|_{L^2(\pd X)}^2.
\end{equation}
Then there is $C>0$ such that for the solution $u$ to the equation \eqref{1l2} to \eqref{1l4}, 
\begin{equation}\label{1l7}
E(u,t)^\frac12\le Ct^{-\frac1{2+\delta}}(\|u_0\|_{H^2(\Omega)}+\|u_1\|_{H^1(\Omega)}),
\end{equation}
either if $Y$ is compact, or if $0<c_0\le b/a_0$ on $\Gamma$.
\end{theorem}

\begin{figure}
\begin{tikzpicture}[minimum size=0.01cm]
\draw [fill=gray!50](0, 0.85) arc (90:270: 0.85) -- (0.1, -0.85) -- (0.1, -0.75) -- (0, -0.75) arc (270: 90: 0.75) -- (0.1, 0.75) -- (0.1, 0.85) -- cycle;

\draw [fill=gray!50] (2, 0.85) arc (90:-90: 0.85) -- (1.9, -0.85) -- (1.9, -0.75) -- (2, -0.75) arc (-90: 90: 0.75) -- (1.9, 0.75) -- (1.9, 0.85) -- cycle;

\draw (0, -0.8) -- (2, -0.8);
\draw (0, 0.8) -- (2, 0.8);
\draw [dashed] (0, -0.8) -- (0, 0.8);
\draw [dashed] (2, -0.8) -- (2, 0.8);
\draw (2, -0.8) arc (-90:90: 0.8);
\draw (0, 0.8) arc (90:270: 0.8);
\node at (1, 0) {$R$};
\node at (-0.4, 0) {$W$};
\node at (2.4, 0) {$W$};

\end{tikzpicture}\hspace{5em}
\begin{tikzpicture}[minimum size=0.01cm]
  \draw[domain=0:1, variable=\x, samples=200] plot ({\x}, {0.7*(sin(deg(pi*\x))/pi+sin(deg(3*pi*\x))/(6*pi)+sin(deg(9*pi*\x))/(36*pi)+sin(deg(27*pi*\x))/(216*pi)+sin(deg(81*pi*\x))/(1296*pi))});
  \draw plot [smooth] coordinates {(0.99,0.012) (1.01, -29/1600) (1.1, -0.2) (1.15 , -0.5) (1.1, -0.8) (1.01, -1+29/1600) (0.99, -1.012)};
  \draw[domain=0:1, variable=\x, samples=200] plot ({\x}, {-1-0.7*((sin(deg(pi*\x))/pi+sin(deg(3*pi*\x))/(6*pi)+sin(deg(9*pi*\x))/(36*pi)+sin(deg(27*pi*\x))/(216*pi)+sin(deg(81*pi*\x))/(1296*pi)))});
  \draw plot [smooth] coordinates {(0.01,0.012) (-0.01, -29/1600) (-0.1, -0.2) (-0.15 , -0.5) (-0.1, -0.8) (-0.01, -1+29/1600) (0.01, -1.012)};
  
  \draw [dashed] (0.5,-0.5) circle (0.8);

  \node at (0.5,-0.5) {$X$};
  \node at (1.8, -0.46) {$X^\sharp$};
\end{tikzpicture}
\caption{Two examples of $X$ when there is no geometric control: the Bunimovich stadium on the left, and a domain with $C^1$-boundary on the right.} 
\label{f2}
\end{figure}
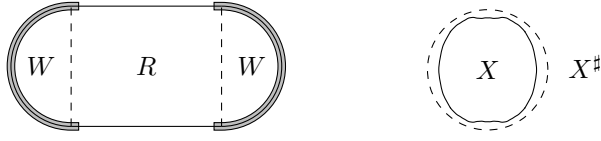
We present some immediate corollaries in cases where we do not have the geometric control condition on the cross-section $X$. Recall the Bunimovich stadium is a planar domain which is a union of a rectangle $R$ with two wings $W$: see Figure \ref{f2}, and \cite{bhw07,nis13}. We consider some damping $a_0(x)>0$ on $\overline{\pd W}$ which vanishes on most of $\pd R$. The geometric control is violated on the Bunimovich stadium $X$, however we have obtained polynomial decay rates. 
\begin{corollary}\label{1t1}
Let $X=R\cup W$ be the Bunimovich stadium, and $a_0(x)>0$ on $\overline{\pd W}$. Then we have 
\begin{equation}\label{1l8}
E(u,t)^\frac12\le Ct^{-\frac1{4}}(\|u_0\|_{H^2(\Omega)}+\|u_1\|_{H^1(\Omega)}),
\end{equation}
for the equation \eqref{1l2} to \eqref{1l4}, where $\Omega=X\times Y$, either if $Y$ is compact, or if $0<c_0\le b/a_0$ on $\Gamma$.
\end{corollary}

Another case we consider is when the boundary of $X$ is merely $C^1$, rough enough to prohibit generalised bicharacteristics from being well-defined. See Figure \ref{f2} and \cite{cv02}. 
\begin{corollary}\label{1t2}
Let $X^\sharp\subset \mathbb{R}^d$ be a smooth bounded connected domain with a smooth boundary $\pd X^\sharp$, equipped with a smooth metric. Let $X\subset X^\sharp$ be a connected domain of the same dimension, with a $C^1$-boundary $\pd X$ and $\pd X\cap \pd X^\sharp=\emptyset$. Assume there is $T>0$ such that every geodesic in $X^\sharp$ hits $\pd X^\sharp$, not necessarily non-diffractively, within time $T$. Let $a_0(x)>0$ on $\pd X$ and we have 
\begin{equation}\label{1l9}
E(u,t)^\frac12\le Ct^{-\frac1{3}}(\|u_0\|_{H^2(\Omega)}+\|u_1\|_{H^1(\Omega)}),
\end{equation}
for the equation \eqref{1l2} to \eqref{1l4}, where $\Omega=X\times Y$, either if $Y$ is compact, or if $0<c_0\le b/a_0$ on $\Gamma$.
\end{corollary}
\begin{remark}
We can use the $D(\Delta)$-norm defined in Section 4, instead $H^2$-norm of $u_0$ in \eqref{1l6}, \eqref{1l7}, \eqref{1l8}, \eqref {1l9}. When $a, b\in C^{0,1/2+\delta}(\Gamma)$ for some $\delta>0$, those two norms are equivalent. 
\end{remark}

The paper is organised in the following manner: in Section 2 we deal with the Geometric Control Condition and resolvent estimates on $X$; in Section 3 we use the resolvent estimates on $X$ to derive those on $\Omega=X\times Y$; in Section 4 we set up the evolution problem and finish the proofs of the results. 

\subsection{Acknowledgement}
The author is grateful to Jared Wunsch for numerous discussions around these results as well as many valuable comments on the manuscript. The author is grateful to Mikko Salo for kindly pointing out the reference for a unique continuation result used in Lemma \ref{2t2}. 

\section{Transverse Resolvent Down Through Zero Frequency}
In this section, the inner products are defaulted to $L^2(X)$ or the $L^2$-space over the indicated space. We firstly show that the geometric control condition gives an high frequency estimate. 
\begin{lemma}(Geometric Control to High Frequency Estimate)\label{2t1}
Assume $\{a_0(x)>0\}$ geometrically controls $X$. Then there exists $C, \mu_0>0$ such that uniformly for any $\abs{\mu}>\mu_0$, $f\in L^2(X)$, $g\in \sqrt{a_0}L^2(\pd X)$, $u$ that 
\begin{gather}\label{2l5}
(\Delta -\mu^2)u=f,\text{ on }X,\\
i\pd_n u+a_0\mu u=g,\text{ on }\pd X,
\end{gather}
we have
\begin{equation}
\|u\|_{L^2(X)}^2+\langle \mu\rangle^{-2}\|\nabla u\|_{L^2(X)}\le C\langle \mu\rangle^{-2}\|f\|_{L^2(X)}^2+C\langle \mu\rangle^{-2}\|g/\sqrt{a_0}\|_{L^2(\pd X)}^2.
\end{equation}
In other words, we have the estimate $\eqref{1l1}$ with $\delta=0$ for the equation \eqref{1l3} to \eqref{1l5}. 
\end{lemma}
\begin{proof}
The idea of the proof follows \cite[Theorem 1.8]{bsw16}. Consider $v_\mu$ being the unique solution to 
\begin{gather}\label{2l4}
(\pd_t^2+\Delta_{x})v_\mu=e^{-i\mu t}\phi(t) f=F,\text{ on } X,\\
(\pd_n+a_0\pd_t)v_\mu=e^{-i\mu t}\phi(t) g=G,\text{ on } \pd X,\\
v_\mu(t,x)=0, \ t<0,
\end{gather}
where $\phi$ is smooth and compactly supported in $(0,1)$ with $\int\phi=1$. Pair \eqref{2l4} with $\pd_t v_\mu$ in $L^2(X)$ to see 
\begin{multline}
\pd_t(\|\nabla v_\mu\|_X^2+\|\pd_t v_\mu\|_{X}^2)=-2\|\sqrt{a_0}\pd_t v_\mu\|_{\pd X}^2\\
+\cre{(\langle G,\pd_t v_\mu\rangle_{\pd X}+\langle F, \pd_t v_\mu\rangle_X)}\le C\|F\|^2_X+C\|G/\sqrt{a_0}\|_{\pd X}^2+C\|\pd_t v_\mu\|_X^2.
\end{multline}
Apply the Gr\"{o}nwall's inequality to see, at $t=1$, 
\begin{equation}
\|\nabla v_\mu\|_X^2+\|\pd_t v_\mu\|_{X}^2\le C\|F\|_{[0,1]\times X}^2+C\|G/\sqrt{a_0}\|_{[0,1]\times \pd X}^2. 
\end{equation}
Since $\{a_0(x)>0\}$ geometrically controls $X$, by \cite[Theorem 5.5]{blr92} we have the exponential decay for all $t>1$:
\begin{equation}\label{2l11}
\|\nabla v_\mu\|_X^2+\|\pd_t v_\mu\|_{X}^2\le Ce^{-ct}(\|f\|_{X}^2+\|g/\sqrt{a_0}\|_{\pd X}^2),
\end{equation}
where the constants do not depend on $\mu$. Inverse Fourier transform $t$ to $\tau$ to observe
\begin{gather}
(\Delta-\tau^2)\mathcal{F}^{-1}v_\mu=\hat{\phi}(\mu-\tau)f,\\
(i\pd_n+a_0 \tau)\mathcal{F}^{-1}v_\mu=i\hat{\phi}(\mu-\tau) g.
\end{gather}
Use $\|\mathcal{F}^{-1}v_\mu\|_{L^2(X)}\le \|v_\mu\|_{L^2(\mathbb{R}_t, L^2(X))}$ to deduce from \eqref{2l11}
\begin{equation}
\abs{\tau}\|\mathcal{F}^{-1}v_\mu\|_X+\|\nabla \mathcal{F}^{-1}v_\mu\|_X\le C(\|f\|+\|g/\sqrt{a_0}\|).
\end{equation}
for all $\tau\in\mathbb{R}$. Now let $\tau=\mu$ to see the solutions to \eqref{2l5} $u_\mu=\mathcal{F}^{-1}v_\mu$, and the desired estimate for $\abs{\mu}\ge \mu_0>0$ uniformly. 
\end{proof}
We here use the unique continuation principles to fix the low frequencies. 
\begin{lemma}[All Frequency Estimate]\label{2t2}
Suppose $a_0\not\equiv 0$ on $\pd X$, and for the equation \eqref{1l3} to \eqref{1l5}, we have the estimate \eqref{1l1} with $\delta\ge 0$, uniformly for all $\abs{\mu}\ge \mu_0>0$. Then for another equation with $f\in L^2(X)$, $g\in \sqrt{a_0}L^2(\pd X)$, 
\begin{gather}\label{2l1}
(\Delta -\mu^2)u=f,\text{ on }X,\\
\label{2l3}
i\pd_n u+a_0\mu u+ia_0 u=g,\text{ on }\pd X,
\end{gather}
uniformly for all $\mu\in \mathbb{R}$ we have
\begin{equation}\label{2l2}
\|u\|_{L^2(X)}^2+\langle \mu\rangle^{-2}\|\nabla u\|_{L^2(X)}\le C\langle \mu\rangle^{2\delta}\|f\|_{L^2(X)}^2+C\langle \mu\rangle^{-2+\delta}\|g/\sqrt{a_0}\|_{L^2(\pd X)}^2.
\end{equation}
Note the perturbative term $ia_0 u$ added to the boundary that makes \eqref{2l1} different from \eqref{1l3}. 
\end{lemma}
\begin{proof}
1. We claim \eqref{2l2} holds for $\abs{\mu}\ge \mu_0>0$. Firstly pair \eqref{2l1} with $u$ and take the imaginary part to see
\begin{equation}
\|\sqrt{a_0}u\|_{\pd X}^2=\mu^{-1}\cre\langle g,u\rangle_{\pd X}-\mu^{-1}\cim\langle f,u\rangle.
\end{equation}
Estimate 
\begin{equation}
\|\sqrt{a_0}u\|_{\pd X}^2\le C\epsilon^{-1}\mu^{-4+\delta}\|f\|^2+C\mu^{-2}\|g/\sqrt{a_0}\|_{\pd X}^2+\epsilon\mu^{2-\delta}\|u\|^2.
\end{equation}
Rewrite the boundary condition \eqref{2l3} as $i\pd_n u+a_0\mu u=g-ia_0 u$ and apply the resolvent estimate \eqref{1l1} to see
\begin{multline}
\|u\|_{L^2(X)}^2+\langle \mu\rangle^{-2}\|\nabla u\|_{L^2(X)}\le C\langle \mu\rangle^{2\delta}\|f\|_{L^2(X)}^2+C\langle \mu\rangle^{-2+\delta}\|\sqrt{a_0} u\|_{\pd X}^2\\
+C\langle \mu\rangle^{-2+\delta}\|g/\sqrt{a_0}\|_{L^2(\pd X)}^2\le C(\langle \mu\rangle^{2\delta}+\langle \mu\rangle^{-6+2\delta})\|f\|_{L^2(X)}^2\\+C(\langle \mu\rangle^{-2+\delta}+\langle \mu\rangle^{-4+\delta})\|g/\sqrt{a_0}\|_{L^2(\pd X)}^2+\epsilon \|u\|^2.
\end{multline}
Absorb the last term to observe the claim. 

2. Assume there exists a sequence $\mu_n\rightarrow \mu$, for some $\abs{\mu}<\mu_0$ such that \eqref{2l2} does not hold, that is, there are $\|u_n\|_{L^2}^2+\langle \mu_n\rangle^{-2}\|\nabla u\|^2\equiv 1$ and $f_n=o_{L^2}(1)$, $g_n=\sqrt{a_0}o_{L^2(\pd X)}(1)$ such that
\begin{gather}\label{2l10}
(\Delta -\mu_n^2)u_n=f_n,\text{ on }X,\\
i\pd_n u_n+a_0 \mu_n u_n+ia_0 u_n=g_n,\text{ on }\pd X. 
\end{gather}
Note that this implies that $\|u_n\|$ uniformly bounded in $H^1$. Pair \eqref{2l10} with $u_n$ and take the real part to see $\|\nabla u_n\|=\abs{\mu}\|u_n\|+o(1)$ and hence $\|u_n\|=(1+\mu^2/\langle \mu\rangle^2)^{-1}+o(1)$. Via a subsequence we have $u_n\rightarrow u\in H^1$ weakly in $H^1$. As $X$ is compact, $u_n\rightarrow u$ in $H^{3/4}$-norm and thus $u_n\rightarrow u$ in $L^2(\pd X)$-norm. Now for any $\phi\in H^1_0$ we have
\begin{equation}
\langle f_n,\phi\rangle=\langle \nabla u_n,\nabla \phi\rangle-\mu_n^2\langle u_n,\phi\rangle+\langle i g_n-ia_0\mu_n u_n+a_0 u_n,\phi\rangle_{\pd X}
\end{equation}
Take $n\rightarrow\infty$ to see
\begin{equation}
\langle \nabla u,\nabla \phi\rangle-\mu^2\langle u,\phi\rangle+\langle a_0 u,\phi\rangle_{\pd X}=0.
\end{equation}
Therefore $u\in H^1$ is the weak solution to
\begin{gather}\label{2l4}
(\Delta-\mu^2) u=0,\text{ on }X,\\
i\pd_n u+ia_0 u=0,\text{ on }\pd X. 
\end{gather}
Pair \eqref{2l4} with $u$ and take the real part to see $\|\sqrt{a_0}u\|_{\pd X}=0$. Hence over $\{a_0>0\}$, $u=0$ and hence $\pd_n u=0$. This implies $u=0$ on $X$ via the classical unique continuation: see for example, \cite[Theorem 1.9]{arrv09}. Thus $u_n\rightarrow 0$ weakly in $H^1$ and in $L^2$-norm, but this contradicts that $\|u_n\|\rightarrow (1+\mu^2/\langle \mu\rangle^2)^{-1}>0$. 
\end{proof}

We now show the main resolvent estimate in this section, for an overdamped interior impedance problem on $X$. 
\begin{proposition}[Transverse Resolvent Down Through Zero Frequency]
Assume the resolvent estimate \eqref{1l1} of parameter $\delta\ge 0$ holds on $X$, and let $b_0(x)\in L^\infty(\pd X)$ be that $0\le b_0\le a_0$ almost everywhere. Then for each $\lambda_0>0$, there exists $C>0$ such that uniformly for any $\mu, \lambda$, $f\in L^2(X)$, $g\in \sqrt{a_0}L^2(\pd X)$, $u$ that
\begin{gather}
(\Delta -z)u=f,\text{ on }X,\\
i\pd_n u+a_0\lambda u+ib_0 u=g,\text{ on }\pd X,
\end{gather}
where $\abs{\lambda}\ge \lambda_0$, $-\infty<z\le\lambda^2$, we have
\begin{gather}\label{2l9}
\|u\|_{H^1(X)}^2\le C\langle \lambda\rangle^{2+2\delta}\|f\|_{L^2(X)}^2+C\langle \lambda\rangle^{\delta} \|g/\sqrt{a_0}\|_{L^2(\pd X)}^2.
\end{gather}
If we assume $0<c_0\le b_0/a_0$ on $\pd X$, we could take $\lambda_0=0$, that is, the same resolvent estimates hold for any $\lambda\in \mathbb{R}$.
\end{proposition}

\begin{proof}
1. Assume \eqref{2l9} is not true, that there exists $\abs{\lambda_n}\ge \lambda_0>0$, $-\infty<z_n\le \lambda_n^2$, $\|u_n\|_{H^1}\equiv 1$, $f_n=o_{L^2}(\langle \lambda_n\rangle^{-1-\delta})$, $g_n=\sqrt{a_0}o_{\pd X}(\langle \lambda_n\rangle^{-\delta/2})$ such that 
\begin{gather}\label{2l6}
(\Delta -z_n)u_n=f_n,\text{ on }X,\\
\label{2l12}
i\pd_n u_n+a_0\lambda_n u_n+ib_0 u_n=g_n,\text{ on }\pd X.
\end{gather}
We firstly establish some a priori estimates. Pair \eqref{2l6} with $u_n$ and take the real and imaginary parts to see
\begin{gather}\label{2l8}
\|\nabla u_n\|^2-z_n\|u_n\|^2+\|\sqrt{b_0}u_n\|_{\pd X}^2=\cre{\langle f_n,u_n\rangle}-\cim{\langle g_n,u_n\rangle_{\pd X}},\\
\|\sqrt{a_0}u_n\|_\pd^2=\lambda_n^{-1}\cre\langle g_n,u_n\rangle_{\pd X}-\lambda_n^{-1}\cim\langle f_n,u_n\rangle,
\end{gather}
where we have divided the imaginary part by $\lambda_n$ as $\lambda_n$ is uniformly bounded away from $0$. Furthermore
\begin{equation}
\|\sqrt{a_0}u_n\|_{\pd X}^2\le C\epsilon^{-1}\lambda_n^{-2}\|g_n/\sqrt{a_0}\|_{\pd X}^2+C\abs{\lambda_n}^{-1}\abs{\langle f_n,u_n\rangle}+\epsilon\|\sqrt{a_0}u_n\|_{\pd X}^2,
\end{equation}
and the absorption of the last term leads to
\begin{equation}\label{2l7}
\|\sqrt{a_0}u_n\|_{\pd X}^2\le C\lambda_n^{-2}\|g_n/\sqrt{a_0}\|_{\pd X}^2+C\lambda_n^{-1}\abs{\langle f_n,u_n\rangle}=o(\langle \lambda_n\rangle^{-2-\delta}),
\end{equation}
because $\langle \lambda_n\rangle\le \abs{\lambda_n}\langle \lambda_0^{-1}\rangle$.

2a. Assume that $\{z_n\}$ has infinitely many non-negative elements, and via a subsequence we have $z_n\ge 0$ for all $n$. Let $\mu_n=\sqrt{z_n}\le \abs{\lambda_n}$. Note that the equation \eqref{2l6} to \eqref{2l12} is equivalent to
\begin{gather}
(\Delta -\mu_n^2)u_n=f_n,\text{ on }X,\\
i\pd_n u_n+a_0\mu_n u_n+ia_0 u_n=g_n',\text{ on }\pd X,\\
g_n'=g_n+(\mu_n-\lambda_n+i)a_0 u_n-ib_0 u_n,
\end{gather}
by reformulating the boundary condition. Note that $b_0\le C a_0$ over $\pd X$ and estimate
\begin{equation}
\|g_n'/\sqrt{a_0}\|_{\pd X}^2\le \|g_n/\sqrt{a_0}\|_{\pd X}^2+C\lambda_n^2\|\sqrt{a_0}u_n\|_{\pd X}^2=o(\langle \lambda_n\rangle^{-\delta})
\end{equation}
by \eqref{2l7}. Apply the estimate \eqref{2l2} given in Lemma \ref{2t2} to see
\begin{gather}
\langle \mu_n\rangle^{-2}=\langle \mu_n\rangle^{-2}\|u_n\|_{H^1}^2\le\|u_n\|^2+\langle \mu_n\rangle^{-2}\|\nabla u_n\|\le C\langle \mu_n\rangle^{2\delta}\|f_n\|^2\\
+C\langle \mu_n\rangle^{-2+\delta}\|g_n'/\sqrt{a_0}\|_{\pd X}^2=o(\langle \mu_n\rangle^{2\delta}\lambda_n^{-2-2\delta}+\langle \mu_n\rangle^{-2+\delta}\lambda_n^{-\delta})=o(\langle \mu_n\rangle^{-2})
\end{gather}
since $1\le\langle \mu_n\rangle\le \langle \lambda_n\rangle$. This is the desired contradiction. 

2b. By excluding the possibility 2a, via a subsequence we have $z_n<0$ for all $n$. Assume that there exists a uniform bound $-1\le -c<0$ that $z_n\le-c<0$ for all $n$. The equation \eqref{2l8} gives
\begin{equation}
\|\nabla u_n\|^2+\abs{z_n}\|u_n\|^2\le \abs{\langle f_n, u_n\rangle}+\abs{\langle g_n/\sqrt{a_0}, \sqrt{a_0}u_n\rangle_{\pd X}}=o(1).
\end{equation}
However 
\begin{equation}
\|\nabla u_n\|^2+\abs{z_n}\|u\|^2\ge c\|u_n\|_{H^1}^2=c>0
\end{equation}
gives the desired contradiction. 

2c. The only possibility left is that via a subsequence $z_n\rightarrow 0$ with $z_n<0$. The equation \eqref{2l8} gives 
\begin{equation}
\|\nabla u_n\|^2\le \abs{\langle f_n, u_n\rangle}+\abs{\langle g_n/\sqrt{a_0}, \sqrt{a_0}u_n\rangle_{\pd X}}=o(1).
\end{equation}
As $\{u_n\}$ are normalised in $H^1$, we know $\|u_n\|=1+o(1)$ and there exists $u\in H^1$ such that $u_n\rightarrow u$ weakly in $H^1$. Then $\nabla u_n\rightarrow \nabla u$ weakly in $L^2$, but $\nabla u_n\rightarrow 0$ in $L^2$, thus $\nabla u\equiv 0$, and $u$ is equal to a constant. Moreover as $X$ is compact, $u_n\rightarrow u$ in both $L^2$ and $L^2(\pd X)$ and \eqref{2l7} implies $u\equiv0$ over the support of $a_0(x)$. Then $u\equiv 0$ over $X$, but $\|u\|=\lim \|u_n\|=1$, yielding the contradiction. 

3. Now we assume further that $c_0\le b_0/a_0$ on $\pd X$ and show \eqref{2l9} for all $\lambda\in\mathbb{R}$. Assume now we have $\lambda_n\rightarrow 0$, then $f=o_{L^2}(1)$, $g=\sqrt{a_0}o_{\pd X}(1)$. Note that we could not use \eqref{2l7} for it fails as $\lambda_0=0$. From \eqref{2l8} we could estimate
\begin{multline}
\|\nabla u_n\|^2+\|\sqrt{b_0}u_n\|_{\pd X}^2\le (\lambda_n^2+z_n) \|u_n\|^2+\|f_n\|\|u_n\|+C\epsilon^{-1}\|g_n/\sqrt{a_0}\|_{\pd X}^2\\
+\epsilon\|\sqrt{a_0}u_n\|_{\pd X}^2.
\end{multline}
Since $c_0 a_0\le b_0$ on ${\pd X}$ and $z_n\le \lambda_n^2$, we have
\begin{equation}
\|\nabla u_n\|^2+\|\sqrt{a_0} u_n\|_{\pd X}^2\le C\lambda_n^2 \|u_n\|^2+C\|f_n\|\|u_n\|+C\|g_n/\sqrt{a_0}\|_{\pd X}^2=o(1),
\end{equation}
and $\|u_n\|=1+o(1)$. Follow the argument in Step 2c to obtain that $u_n\rightarrow 0$ in $L^2$ and hence the contradiction. 
\end{proof}

\section{Resolvent Estimates on Product Manifolds}
Let $H_\pd^s(Y)$ be the closure of
\begin{equation}
\left\{u\in C^\infty(Y)\cap H^s(Y): u \text{ satisfies the chosen boundary condition on }\pd Y\right\}
\end{equation}
under $H^s(Y)$-norm. Let $H_\pd^s(\Omega)$, the Sobolev space of order $s$ on $\Omega$ with the chosen boundary condition on $X\times \pd Y$ be the closure of
\begin{multline}
\left\{u\in C^\infty(\Omega)\cap H^s(\Omega): \forall x\in X,\right.\\
\left. u(x,\cdot) \text{ satisfies the chosen boundary condition on }\pd Y\right\}
\end{multline}
under $H^s(\Omega)$-norm. In this section, the inner products are defaulted to $L^2(\Omega)$ or the $L^2$-space over the indicated space. 
\begin{proposition}[Resolvent Estimates on Product Manifolds]
On $\Omega=X\times Y$, consider a solution $u\in H_\pd^1(\Omega)$ to the interior impedance problem,
\begin{gather}\label{3l1}
(\Delta_x+\Delta_y-\lambda^2)u(x,y)=f, \text{ on }\Omega,\\
\label{3l2}
(i\pd_n+a_0(x)\lambda+ib_0(x))u(x,y)=g, \text{ on }\Gamma=\pd X\times Y,
\end{gather}
where $f\in L^2(\Omega)$, $g\in \sqrt{a_0}L^2(\Gamma)$, and $b_0(x)$ is some $L^\infty(\pd X)$-function that $0\le b_0\le a_0$ almost everywhere. Assume the resolvent estimate \eqref{1l1} of parameter $\delta\ge 0$ holds on $X$. Then for any $\lambda_0>0$, we have some $C>0$ independent of $\lambda, u, f, g$ such that for any $\abs{\lambda}\ge\lambda_0$,
\begin{equation}\label{3l3}
\|u\|_{L^2(\Omega)}^2+\langle \lambda\rangle^{-2}\|\nabla u\|_{L^2(\Omega)}^2\le C\langle \lambda\rangle^{2+2\delta}\|f\|_{L^2(\Omega)}^2+C\langle \lambda\rangle^{\delta}\|g/\sqrt{a_0}\|_{L^2(\Gamma)}^2.
\end{equation}
If $f\in H^1(\Omega)$ and $g\in \sqrt{a_0}H^{1}(\Gamma)$, then
\begin{multline}\label{3l9}
\|u\|_{L^2(\Omega)}^2+\langle \lambda\rangle^{-2}\|\nabla u\|_{L^2(\Omega)}^2\le C\langle \lambda\rangle^{2\delta}\|f\|_{H^1(\Omega)}^{2}+C\langle \lambda\rangle^{-2+\delta}C\|g/\sqrt{a_0}\|_{H^1(\Gamma)}^2.
\end{multline}
If we assume $0<c_0\le b_0/a_0$ on $\pd X$, we could take $\lambda_0=0$, that is, the same resolvent estimates hold for any $\lambda\in \mathbb{R}$.
\end{proposition}
\begin{remark}
This is the boundary impedance version of \cite[Theorem 6.1]{bz04}. 
\end{remark}
\begin{proof}
1. We show some a priori estimates on $\Omega$. Pair $(\Delta-\lambda^2) u$ against $u$ and take the real and imaginary parts to see
\begin{gather}\label{3l6}
\|\nabla u\|^2-\lambda^2\|u\|^2+\|\sqrt{b_0}u\|_{\Gamma}^2=\cre{\langle f,u\rangle}-\cim{\langle g,u\rangle_{\Gamma}},\\
\|\sqrt{a_0}u\|_{\Gamma}^2=\lambda^{-1}\cre\langle g,u\rangle_{\Gamma}-\lambda^{-1}\cim\langle f,u\rangle.
\end{gather}
Estimate
\begin{equation}
\|\sqrt{a_0}u\|_{\Gamma}^2\le C\epsilon^{-1}\lambda^{-3}\|f\|^2+C\lambda^{-2}\|g/\sqrt{a_0}\|_{\Gamma}^2+\epsilon \lambda \|u\|^2,
\end{equation}
and
\begin{multline}
\abs{\|\nabla u\|^2-\lambda^2\|u\|^2}\le C\epsilon^{-1}\lambda^{-2}\|f\|^2+C\lambda^{-1}\|g/\sqrt{a_0}\|_{\Gamma}^2+C\lambda\|\sqrt{a_0}u\|_{\Gamma}^2\\
\le C\epsilon^{-1}\lambda^{-2}\|f\|^2+C\lambda^{-1}\|g\|_{\Gamma}^2+\epsilon \lambda^2\|u\|^2.
\end{multline}
This implies for $\abs{\lambda}\ge \lambda_0>0$ we have
\begin{gather}\label{3l8}
\|u\|^2\le C\epsilon^{-1}\lambda^{-4}\|f\|^2+C\lambda^{-3}\|g/\sqrt{a_0}\|_{\Gamma}^2+C\lambda^{-2}\|\nabla u\|^2,\\
\label{3l5}
\|\nabla u\|^2\le C\epsilon^{-1}\lambda^{-2}\|f\|^2+C\lambda^{-1}\|g/\sqrt{a_0}\|_{\Gamma}^2+C\lambda^2\|u\|^2.
\end{gather}
When $c_0\le b_0/a_0$, directly estimating \eqref{3l6} leads to 
\begin{equation}\label{3l7}
\|\nabla u\|\le \langle \lambda\rangle^2\|u\|^2+C\|f\|^2+\|g/\sqrt{a_0}\|_\Gamma^2
\end{equation}
for any $\lambda\in\mathbb{R}$.

2. We show \eqref{3l3}. As $Y$ is complete and $\Delta_y$ is essentially self-adjoint, we have the uniquely determined spectral resolution of $\Delta_y$ on $L^2(Y)$ that 
\begin{equation}
\Delta_y v=\int_0^\infty \rho^2\ dE_\rho (v),
\end{equation}
where $E_\rho$ is a projection-valued measure and $\supp(E_\rho)\subset[0,\infty)$. Fix some $\epsilon<1$ small. We will be using the spectral projector
\begin{equation}
\Pi_{\eta, \lambda} v=\int_{\eta-\epsilon /\langle \eta\rangle\langle \lambda\rangle^{1+\delta}}^{\eta+\epsilon /\langle \eta\rangle\langle \lambda\rangle^{1+\delta}}\ dE_\rho(v)
\end{equation}
which localises an $L^2(Y)$-function into a small spectral window of width $2\epsilon/\langle \eta\rangle\langle \lambda\rangle^{1+\delta}$ near $\eta$. Apply the spectral projector to the equation \eqref{3l1} and \eqref{3l2} to see
\begin{gather}
(\Delta_x-(\lambda^2-\eta^2))\Pi_{\eta, \lambda}u(x,y)=\Pi_{\eta, \lambda}f-\Pi_{\eta, \lambda}(\Delta_y -\eta^2)u, \text{ on }\Omega,\\
(i\pd_n+a_0(x)\lambda+ib_0(x))\Pi_{\eta, \lambda}u(x,y)=\Pi_{\eta, \lambda}g, \text{ on }\Gamma.
\end{gather}
Since for $\rho\in [\eta-\epsilon /\langle \eta\rangle\langle \lambda\rangle^{1+\delta},\eta+\epsilon /\langle \eta\rangle\langle \lambda\rangle^{1+\delta}]$,  
\begin{equation}
\abs{\rho^2-\eta^2}\le 2\epsilon\abs{\eta}/\langle \eta\rangle\langle \lambda\rangle+4\epsilon^2/\langle \eta\rangle^2\langle \lambda\rangle^2\le 6 \epsilon/\langle \lambda\rangle^{1+\delta},
\end{equation}
we have
\begin{multline}
\|\Pi_{\eta, \lambda}(\Delta_y -\eta^2)u\|=\left\|\int_{\eta-\epsilon /\langle \eta\rangle\langle \lambda\rangle^{1+\delta}}^{\eta+\epsilon /\langle \eta\rangle\langle \lambda\rangle^{1+\delta}} (\rho^2-\eta^2)\ dE_\rho(u)\right\| \\ \le 6\epsilon\langle \lambda\rangle^{-1-\delta}\|\Pi_{\eta,\lambda} u\|.
\end{multline}
Apply the resolvent estimate \eqref{2l9} on $X$ to see uniformly for $y, \eta$ and $\lambda\ge\lambda_0$, 
\begin{multline}
\|\Pi_{\eta,\lambda}u\|_{H^1(X)}^2\le C\langle \lambda\rangle^{2+2\delta}\|\Pi_{\eta,\lambda}f\|_{L^2(X)}^2+C\langle \lambda\rangle^{\delta}\|\Pi_{\eta,\lambda}g/\sqrt{a_0}\|_{L^2(\pd X)}^2\\
+36C\epsilon^2 \|\Pi_{\eta,\lambda}u\|^2.
\end{multline}
Now choose $\epsilon$ small to absorb the last term to have
\begin{equation}\label{3l4}
\|\Pi_{\eta,\lambda}u\|_{H^1(X)}^2\le C\langle \lambda\rangle^{2+2\delta}\|\Pi_{\eta,\lambda}f\|_{L^2(X)}^2+C\langle \lambda\rangle^{\delta}\|\Pi_{\eta,\lambda}g/\sqrt{a_0}\|_{L^2(\pd X)}^2.
\end{equation}
Since the choice of $\epsilon$ does not depend on $y, \eta, \lambda$, the constants in this estimate are again independent of $y, \eta, \lambda$. Now fix $\lambda$ and consider $\{(\eta-\epsilon /\langle \eta\rangle\langle \lambda\rangle^{1+\delta},\eta+\epsilon /\langle \eta\rangle\langle \lambda\rangle^{1+\delta})\}_{\eta}$ as a cover of $[0,\infty)$. It has a locally finite subcover of multiplicity 3. Sum up \eqref{3l4} to see
\begin{equation}
\|u\|_{H^1(X)}^2\le 3C\langle \lambda\rangle^{2+2\delta}\|f\|_{L^2(X)}^2+3C\langle \lambda\rangle^{\delta}\|g/\sqrt{a_0}\|_{L^2(\pd X)}^2.
\end{equation}
Integrating the previous inequality over $Y$ gives
\begin{equation}\label{3l10}
\|u\|_{L^2(\Omega)}^2+\|\nabla_x u\|_{L^2(\Omega)}^2\le 3C\langle \lambda\rangle^{2+2\delta}\|f\|_{L^2(\Omega)}^2+3C\langle \lambda\rangle^\delta\|g/\sqrt{a_0}\|_{L^2(\Gamma)}^2.
\end{equation}
Since $C$ does not depend on $\lambda$, we have the desired estimate for $\|u\|$. Visit \eqref{3l5} to bound $\|\nabla u\|$. Note that when $c_0\le b_0/a_0$, we could choose $\lambda_0=0$ and use \eqref{3l7} instead of \eqref{3l5}. This concludes the proof of \eqref{3l3}. 

3. We show \eqref{3l9}. Let the first-order elliptic pseudodifferential operator $\Lambda$ be defined via
\begin{equation}
\Lambda v=(1+\Delta_y)^\frac12 v=\int_{0}^\infty \langle \rho\rangle \ dE_\rho(v).
\end{equation}
Note that $[\Lambda, \Pi_{\eta,\lambda}]=0$ as they are in the same functional calculus. Apply $\Lambda$ to the equation \eqref{3l1} and \eqref{3l2} to see
\begin{gather}
(\Delta_x-(\lambda^2-\eta^2))\Pi_{\eta, \lambda}\Lambda u=\Pi_{\eta, \lambda}\Lambda f-\Pi_{\eta, \lambda}(\Delta_y -\eta^2)\Lambda u, \text{ on }\Omega,\\
(i\pd_n+a_0(x)\lambda)\Pi_{\eta, \lambda}\Lambda u(x,y)=\Pi_{\eta, \lambda}\Lambda g, \text{ on }\Gamma.
\end{gather}
Following the same reasoning as in the last step we end up with
\begin{equation}
\|\Lambda u\|_{L^2(\Omega)}^2\le 3C\langle \lambda\rangle^{2+2\delta}\|\Lambda f\|_{L^2(\Omega)}^2+3C\langle \lambda\rangle^{\delta}\|(\Lambda g)/\sqrt{a_0}\|_{L^2(\Gamma)}^2.
\end{equation}
Now note $\|u\|_{H^1}\sim \|\Lambda u\|+\|\nabla_x u\|$, and by \eqref{3l10} we have
\begin{equation}\label{3l11}
\|u\|_{H^1(\Omega)}^2\le C\langle \lambda\rangle^{2+2\delta}\|f\|_{H^1(\Omega)}^{2}+C\langle \lambda\rangle^{\delta}\|g/\sqrt{a_0}\|_{H^1(\Gamma)}^2.
\end{equation}
For $\abs{\lambda}\ge \lambda_0>0$ we have from \eqref{3l11} that
\begin{equation}\label{3l12}
\lambda^{-2}\|\nabla u\|^2_\Omega\le C\langle \lambda\rangle^{2\delta}\|f\|_{H^1(\Omega)}^{2}+C\langle \lambda\rangle^{-2+\delta}C\|g/\sqrt{a_0}\|_{H^1(\Gamma)}^2
\end{equation}
and \eqref{3l8} concludes the proof of \eqref{3l9}. Moreover when $c_0\le b_0/a_0$ and for $\lambda$ bounded, \eqref{3l11} implies
\begin{equation}
\|u\|_{L^2(\Omega)}^2\le C\|f\|_{H^1(\Omega)}^{2}+C\|g/\sqrt{a_0}\|_{H^1(\Gamma)}^2.
\end{equation}
With \eqref{3l12} we have \eqref{3l9} when $\lambda_0=0$.
\end{proof}


\section{Evolution and Energy Decay}
Let $\Delta=\Delta_x\otimes \id_y+\id_x\otimes \Delta_y$, and the domain $D(\Delta)$ be that of the minimal closed extension of $\Delta:H^2_\pd(\Omega)\rightarrow L^2(\Omega)$, that is, the closure of $C^\infty(\Omega)\cap H^1_\pd (\Omega)$ in $H^1_\pd (\Omega)$ under the graph norm 
\begin{equation}
\|u\|_{D(\Delta)}^2=\|u\|_{H^1(\Omega)}^2+\|\Delta u\|_{L^2(\Omega)}^2.
\end{equation}
Let $\mathcal{H}=H_\pd^1(\Omega)\oplus L_\pd^2(\Omega)$. If $b\equiv 0$ on $\Gamma$, we further restrict $\mathcal{H}$ to the set of $(\tilde u, \tilde v)\in H_\pd^1(\Omega)\oplus L_\pd^2(\Omega)$ that
\begin{equation}
\int_\Omega \tilde v+\int_\Gamma a(x,y)\tilde u(x,y)=0,
\end{equation}
a codimension-1 subspace of $H_\pd^1\oplus L_\pd^2$. Let the generator of our semigroup be 
\begin{equation}
A=\begin{pmatrix}
0 & \id\\
-\Delta & 0
\end{pmatrix}: D(A)\rightarrow \mathcal{H},
\end{equation}
and the domain is
\begin{equation}\label{4l5}
D(A)=\{(u,v)\in \mathcal{H}: A(u,v)\in \mathcal{H}, \pd_n u+a(x,y) v+b(x,y) u=0 \text{ on } \Gamma\},
\end{equation}
which is equivalent to the set of all $(u,v)\in \mathcal{H}$, $u\in D(\Delta)$, $v\in H^1_\pd(\Omega)$ and $\pd_n u+a(x,y) v+b(x,y) u=0 \text{ on } \Gamma$.
Note $D(A)$ is dense in $\mathcal{H}$, and 
\begin{equation}
\{(u,v)\in H^2_\pd(\Omega)\oplus H^1_\pd(\Omega): \pd_n u+a(x,y) v+b(x,y) u=0 \text{ on } \Gamma\}
\end{equation}
is dense in $D(A)$. Define the energy seminorm $E$ on $\mathcal{H}$ via
\begin{equation}
\|(u,v)\|_E^2=\|\nabla u\|_{L^2(\Omega)}^2+\|v\|_{L^2(\Omega)}^2+\|\sqrt{b}u\|_{L^2(\Gamma)}^2.
\end{equation}
\begin{lemma}
$E$ is a norm on $\mathcal{H}$ and is equivalent to the $H^1\oplus L^2$-norm. $A$ is maximally dissipative with respect to the energy norm $E$ and generates a contraction semigroup $e^{tA}$ on $\mathcal{H}$. 
\end{lemma}
\begin{proof}
See \cite[Proposition 3.1]{wan21}. Note that via a direct computation we have
\begin{equation}\label{4l1}
\cre\langle A(u,v), (u,v)\rangle_E=-\int_\Gamma a(x,y)\abs{v}^2\le 0. 
\end{equation}
This demonstrates the dissipative nature of $A$. 
\end{proof}

\begin{proposition}
Assume the resolvent estimate \eqref{1l1} holds with parameter $\delta\ge 0$. Then $A$ has no purely imaginary spectrum, and there exists $C>0$ that for $\lambda\in\mathbb{R}$.
\begin{equation}\label{4l2}
\|(A+i\lambda)^{-1}\|_{E\rightarrow E}\le C\langle \lambda\rangle^{2+\delta},
\end{equation} 
either if $Y$ is compact, or if $0<c_0\le b/a_0$ on $\Gamma$. 
\end{proposition}
\begin{proof}
1. Fix some $\lambda_0>0$. Assume \eqref{4l2} is not true for $\{\abs{\lambda}\ge \lambda_0\}$. Then there are $\abs{\lambda_n}\ge \lambda_0$ and $(u_n, v_n)\in D(A)$ with 
\begin{equation}
\|(u_n,v_n)\|_E^2=\|\nabla u_n\|^2+\|v_n\|^2+\|\sqrt{b} u_n\|_\Gamma^2\equiv 1
\end{equation}
such that 
\begin{equation}
\left(A+ i\lambda_n\right)(u_n, v_n)=o_E(\langle \lambda_n\rangle^{-2-\delta})=o_\mathcal{H}(\langle \lambda_n\rangle^{-2-\delta}).
\end{equation}
This can be reduced to
\begin{gather}\label{4l4}
\Delta u_n-\lambda_n^2 u_n=o_{L^2}(\langle \lambda_n\rangle^{-2-\delta})+o_{H^1}(\langle \lambda_n\rangle^{-1-\delta}),\\
\label{4l6}
v_n=-i\lambda_n u_n+o_{H^1}(\langle \lambda_n\rangle^{-2-\delta}).
\end{gather}
Now \eqref{4l1} implies
\begin{equation}
\|\sqrt{a} v_n\|^2_\Gamma=-\cre\langle (A+i\lambda_n)(u_n,v_n), (u_n,v_n)\rangle_E=o(\langle \lambda_n\rangle^{-2-\delta}).
\end{equation}
and
\begin{equation}\label{4l3}
\|\sqrt{a}u_n\|_{\Gamma}\le \abs{\lambda_n}^{-1}(\|\sqrt{a}v_n\|_\Gamma+o(\langle \lambda_n\rangle^{-2-\delta}))=o(\langle \lambda_n\rangle ^{-2-\delta/2}).
\end{equation}
Note this is the only place we used $\abs{\lambda_n}\ge\lambda_0>0$. Pair \eqref{4l4} with $u_n$ and take the real part to see
\begin{equation}\label{4l7}
\|\nabla u_n\|^2+\|\sqrt{b}u_n\|_\Gamma^2=\abs{\lambda_n}^2\|u_n^2\|+o(\langle \lambda_n\rangle^{-1-\delta}).
\end{equation}
Use \eqref{4l5} to see
\begin{gather}
\Delta u_n-\lambda_n^2 u_n=o_{L^2}(\langle \lambda_n\rangle^{-2-\delta})+o_{H^1}(\langle \lambda_n\rangle^{-1-\delta}),\\
i\pd_n +a\lambda_n u_n+ib u_n=a o_{H^{1/2}(\Gamma)}(\langle \lambda_n\rangle^{-2-\delta}).
\end{gather}
By Assumption \ref{1t3}, we have $c_0\le a/a_0\le 1$ and
\begin{equation}
i\pd_n u_n+a_0\lambda_n u_n+i b_0 u_n=\lambda_n (a-a_0)u_n-i(b-b_0)u_n+ao_{\Gamma}(\langle \lambda_n\rangle^{-2-\delta})=g_n,
\end{equation}
where $b_0(x)=\inf_Y b(x,y)\ge 0$. Note $a, b\le a_0$ and use \eqref{4l3} to estimate
\begin{equation}
g_n=\sqrt{a_0}o_{L^2(\Gamma)}(\langle \lambda_n\rangle^{-1-\delta/2}).
\end{equation}
Thus we have turned the quasimodes with boundary damping $a(x,y), b(x,y)$ into those with $a_0(x), b_0(x)$:
\begin{gather}
\Delta u_n-\lambda_n^2 u_n=f^0_n+f^1_n, \text{ on }\Omega\\
\label{4l8}
i\pd_n u_n +a_0\lambda_n u_n+ib_0 u_n=\sqrt{a_0} o_{\Gamma}(\langle \lambda_n\rangle^{-1-\delta/2})=\sqrt{a_0} g_n, \text{ on }\Gamma\\
f^0_n=o_{L^2}(\langle \lambda_n\rangle^{-2-\delta}), \ f^1_n=o_{H^1}(\langle \lambda_n\rangle^{-1-\delta}).
\end{gather}

2. By the standard existence theorem, the problem 
\begin{gather}
(\Delta w_n-\lambda_n^2 w_n)=f^0_n,\\
i\pd_n w_n +a_0\lambda_n w_n+ib_0 w_n=\sqrt{a_0} g_n
\end{gather}
has a unique solution $w_n\in D(\Delta)\subset H^1_\pd(\Omega)$: see \cite[Proposition 2.8, Proposition 4.2]{wan21}, \cite[Theorem 4.11]{mcl00} for details. Apply the resolvent estimate \eqref{3l3} to see
\begin{equation}
\|w_n\|^2\le C \langle \lambda_n\rangle^{2+2\delta} \|f^0_n\|^2+C\langle \lambda_n\rangle^\delta\|g_n\|_{\Gamma}^2=o(\langle \lambda_n\rangle^{-2}).
\end{equation}
Meanwhile $u_n-w_n$ satisfies
\begin{gather}
(\Delta (u_n-w_n)-\lambda_n^2 (u_n-w_n))=f^1_n,\\
(i\pd_n +a_0\lambda_n+ib_0)(u_n-w_n)=0.
\end{gather}
Apply \eqref{3l9} to see
\begin{equation}
\|u_n-w_n\|^2\le C\langle \lambda_n\rangle^{2\delta} \|f^1_n\|_{H^1}^2=o(\langle \lambda_n\rangle^{-2}).
\end{equation}
Thus $\|u_n\|=o(\langle \lambda_n\rangle^{-1})$. By \eqref{4l6} and \eqref{4l7} we have
\begin{equation}
1\equiv \|\nabla u_n\|^2+\|v_n\|^2+\|\sqrt{b} u_n\|_\Gamma^2=2\abs{\lambda_n}^2\|u_n\|^2+o(\langle \lambda_n\rangle^{-1-\delta})=o(1),
\end{equation}
which gives the desired contradiction for $\abs{\lambda}\ge \lambda_0>0$. 

3. When $Y$ is compact, we show $0\notin \operatorname{Spec}(A)$. By the Rellich theorem we know $D(A)\hookrightarrow \mathcal{H}$ compactly. Furthermore $\id-A$ is surjective from $D(A)$ to $\mathcal{H}$. Hence $\operatorname{Spec}(A)$ is discrete, and $0\in \operatorname{Spec}(A)$ only if $0$ is an eigenvalue. Indeed, if $A(u,v)=0$, we have $v=0$ and $\Delta u=0$ and then
\begin{equation}
\|(u,v)\|_E^2=\|\nabla u\|^2+\|v\|^2+\|\sqrt{b}u\|_\Gamma^2=0.
\end{equation}
Therefore $(u,v)=0$ since $E$ is point-separating. Then $i(-2\lambda_0, 2\lambda_0)$ is disjoint from $\operatorname{Spec}(A)$ for some $\lambda_0>0$ and we have \eqref{4l2}. 

4. When $Y$ is non-compact and $c_0\le b/a_0$ on $\Gamma$, we have
\begin{equation}
\|\sqrt{a} u_n\|_\Gamma\le C \|\sqrt{b} u_n\|_\Gamma\le \|(u_n, v_n)\|_E\equiv 1
\end{equation}
as $\abs{\lambda_n}\rightarrow 0$ and this fixes \eqref{4l3} for all $\abs{\lambda_n}\ge \lambda_0=0$. Moreover, in \eqref{4l8} we have $b_0\ge c_0 a_0$ and the resolvent estimates \eqref{3l3} and \eqref{3l9} work for $\lambda_0=0$, that is, for all $\lambda$. This concludes the proof. 
\end{proof}
We cite \cite[Theorem 2.4]{bt10} to characterise the polynomial energy decay:
\begin{theorem}[Borichev-Tomilov]\label{4t1}
Let $A$ be a maximal dissipative operator that generates a contraction $C^0$-semigroup in a Hilbert space $\mathcal{H}$ and assume that $i\setr\cap \operatorname{Spec}(A)$ is empty. Fix $k>0$, and the following are equivalent:
\begin{enumerate}[wide]
\item There exists $C>0$ such that for any $\abs{\lambda}>0$ one has
\begin{equation}
\left\|(A-i\lambda)^{-1}\right\|_{\mathcal{H}\rightarrow \mathcal{H}}\le C\abs{\lambda}^k.
\end{equation}
\item There exists $C>0$ such that
\begin{equation}
\left\|e^{tA}A^{-1}\right\|_{\mathcal{H}\rightarrow \mathcal{H}}\le C t^{-\frac{1}{k}}.
\end{equation}
\end{enumerate}
\end{theorem}

\begin{proof}[Proof of Theorem \ref{0t2}] 
The resolvent estimate \eqref{4l2} on the generator $A$ with Theorem \ref{4t1} of Borichev and Tomilov implies
\begin{equation}
E(u, t)^{\frac12}=\|e^{tA}A^{-1}A(u_0, u_1)\|_E\le Ct^{-\frac1{2+\delta}}(\|u_0\|_{H^2}+\|u_1\|_{H^1}),
\end{equation}
as claimed in Theorem \ref{0t2}. 
\end{proof}

\begin{proof}[Proof of Theorem \ref{0t1}]
Lemma \ref{2t1} implies that the geometric control implies the resolvent estimate \eqref{1l1} holds with parameter $\delta=0$. Invoke Theorem \ref{0t2} to see the decay. 
\end{proof}

\begin{proof}[Proof of Corollary \ref{1t1}]
Let $X$ be the Bunimovich stadium. Revisit \cite[Proposition 6]{nis13} to see for the equation \eqref{1l3} to \eqref{1l5} we have for $\abs{\mu}>1$,
\begin{equation}
\|u\|_{L^2(X)}^2\le C\lambda^2\|f\|_{L^2(X)}^2+C\|g\|_{L^2(\pd X)}^2,
\end{equation}
that is, the estimate \eqref{1l1} holds with parameter $\delta=2$. Invoke Theorem \ref{0t2} to get the decay. 
\end{proof}

\begin{proof}[Proof of Corollary \ref{1t2}]
Let $X, X^\sharp$ be described as in the statement. Invoke \cite[Proposition 2.1]{cv02} to see for the equation \eqref{1l3} to \eqref{1l5} we have for $\abs{\mu}\ge\lambda_0>0$,
\begin{equation}
\|u\|_{L^2(X)}^2\le C\lambda^{-2}\|f\|_{L^2(X)}^2+C\lambda^{-1}\|g\|_{L^2(\pd X)}^2,
\end{equation}
that is, the estimate \eqref{1l1} holds with parameter $\delta=1$. Invoke Theorem \ref{0t2} to get the decay. 
\end{proof}

\begin{proposition}\label{4t2}
Assume $a=a_0\in C^{1,1/2+\delta}(\pd X)$ for some $\delta>0$, $a_0(x)>0$ on $\pd X$ and $b\equiv 0$. Then there exist quasimodes $(\lambda_n, U_n)$, $\|U_n\|_\mathcal{H}\equiv 1$ such that $\|(A-i\lambda_n)U_n\|_{\mathcal{H}}=\bigo(\langle \lambda_n\rangle^{-2})$. This implies that the equation \eqref{1l2} to \eqref{1l4} cannot be stable at the rate $t^{-1/(2-\epsilon)}$ for any $\epsilon>0$.
\end{proposition}
\begin{remark}
When $X=\mathbb{B}^d$ and $a(x)=1$, we constructed those quasimodes in \cite[Proposition 3.6, Proposition 4.5]{wan21} using special functions. 
\end{remark}
\begin{proof}
1. We begin by constructing a sequence of quasimodes on $X$. Fix an Dirichlet eigenfunction $w^0$ with respect to some eigenvalue $\mu^2$ where $\mu>0$, that is
\begin{gather}
(\Delta-\mu^2)w^0=0,\text{ on }X,\\
w^0=0,\text{ on } \pd X,
\end{gather}
and without loss of generality $\|w^0\|_X=1$. Now note the H\"{o}lder regularity of $a$ implies $a$ is a Sobolev multiplier on $H^{3/2}(\pd X)$, and being uniformly away from 0 makes $1/a$ one as well: also see \cite[Lemma 2.6]{wan21} and the reference within. Consider an extension $w^1\in H^2(X)$ from the given boundary data:
\begin{equation}
w^1=-i\pd_n w^0/a\in H^{\frac32}(\pd X), \ \pd_n w^1=0, \text{ on } \pd X. 
\end{equation}
Then $w_\lambda=w^0+\lambda^{-1}w^1$ forms a sequence of quasimodes on $X$:
\begin{gather}
(\Delta_x-\mu^2)w_\lambda=\lambda^{-1}(\Delta_x-\mu^2)w^1, \text{ on }X,\\
(i\pd_n+a\lambda)w_\lambda=0,\text{ on }\pd X,
\end{gather}
where $\|w_\lambda\|=1+\bigo(\lambda^{-1})$.

2. Since $\Delta_y$ is unbounded on $L^2(Y)$, there exist $\lambda_n>\mu$ with $\lambda_n\rightarrow \infty$ and $v_n\in L^2(Y)$ such that
\begin{equation}
v_n=\int_{\sqrt{\lambda_n^2-\mu^2}-\lambda_n^{-2}}^{\sqrt{\lambda_n^2-\mu^2}+\lambda_n^{-2}} dE_\rho(v_n), \ \|v_n\|_Y\equiv 1.
\end{equation}
In other words, $v_n$'s are spectrally supported inside the spectral windows $[\sqrt{\lambda_n^2-\mu^2}-\lambda_n^{-2},\sqrt{\lambda_n^2-\mu^2}+\lambda_n^{-2}]$ that shrink and shift to the infinity. Note $\abs{\rho^2-(\lambda_n^2-\mu^2)}=\bigo(\lambda_n^{-1})$ inside this spectral window and this implies
\begin{equation}
(\Delta_y-(\lambda_n^2-\mu^2)) v_n=\int_{\sqrt{\lambda_n^2-\mu^2}-\lambda_n^{-2}}^{\sqrt{\lambda_n^2-\mu^2}+\lambda_n^{-2}} (\rho^2-(\lambda_n^2-\mu^2))\ dE_\rho(v_n)=\bigo_{L^2(Y)}(\lambda_n^{-1}).
\end{equation}
Now let $u_n=w_{\lambda_n}(x)v_n(y)$. We have
\begin{multline}\label{4l9}
(\Delta-\lambda_n^2)u_n=v_n(\Delta_x-\mu^2)w_{\lambda_n}+w_{\lambda_n}(\Delta_y-(\lambda_n^2-\mu^2)) v_n\\
=\lambda_n^{-1}(\Delta_x-\mu^2)w^1+w_{\lambda_n}\bigo_{L^2(Y)}(\lambda_n^{-1})=\bigo_{L^2}(\lambda_n^{-1}),
\end{multline}
where $\|u_n\|=1+\bigo(\lambda_n^{-1})$. Pair \eqref{4l9} with $u_n$ and take the real part to see $\|\nabla u_n\|=\lambda_n+\bigo(1)$. Note $i\pd_n u_n+a_0 \lambda_n u_n=0$ on $\Gamma$ and hence $(u_n, i\lambda_n u_n)\in \mathcal{H}$, and $\|(u_n, i\lambda_n u_n)\|_\mathcal{H}=\sqrt{2}\lambda_n+\bigo(1)$. Let $U_n=(u_n, i\lambda_n u_n)/\|(u_n, i\lambda_n u_n)\|_\mathcal{H}$. Then
\begin{multline}
(A-i\lambda_n)U_n=\|(u_n, i\lambda_n u_n)\|_\mathcal{H}^{-1}(A-i\lambda_n)(u_n, i\lambda_n u_n)\\
=\bigo(\lambda_n^{-1})(0, -(\Delta-\lambda_n^2)u_n)=\bigo(\lambda_n^{-2})=o(\lambda_n^{-2+\epsilon}),
\end{multline}
for any $\epsilon>0$. Invoke Theorem \ref{4t1} to conclude the proof.
\end{proof}

\bibliography{bib}
\bibliographystyle{amsalpha}

\end{document}